\documentclass[reqno]{amsart}
\usepackage{amsthm, amscd, amsfonts, amssymb, graphicx,tikz, color, environ}

\usepackage{paralist}
\usepackage{cite}
\usepackage{bigints}
\usepackage{dsfont}
\usepackage{empheq}
\usepackage{amssymb}
\usepackage{cases}
\usepackage{enumitem}
\allowdisplaybreaks

\usepackage{caption}
\usepackage{booktabs}

\usepackage{float}
\usepackage[ruled,vlined,linesnumbered,noresetcount]{algorithm2e}

\usepackage{hyperref}
\usepackage[top=3cm, bottom=2cm, left=3cm, right=3cm]{geometry}

\theoremstyle{plain}
\newtheorem{theorem}{Theorem}[section]

\newtheorem{lemma}[theorem]{Lemma}

\theoremstyle{remark}
\newtheorem{remark}[theorem]{Remark}
\numberwithin{equation}{section}

\title[Degenerate mean-field games systems]
      {Stability of backward inverse problems for degenerate mean-field game systems}
\author{S. E. Chorfi}
\address{S. E. Chorfi, L. Maniar, Faculty of Sciences Semlalia, Cadi Ayyad University, B.P. 2390, Marrakesh, Morocco}
\email{s.chorfi@uca.ac.ma, maniar@uca.ac.ma}

\author{A. Habbal}
\address{A. Habbal, M. Jahid, L. Maniar, A. Ratnani,\\ The UM6P-Vanguard Center, University Mohammed VI Polytechnic, Benguerir, Morocco}
\email{Abderrahmane.Habbal@um6p.ma, Meryeme.JAHID@um6p.ma, Lahcen.Maniar@um6p.ma, Ahmed.Ratnani@um6p.ma}

\address{A. Habbal, University of Côte d’Azur, Inria, LJAD, Parc Valrose 06108 Nice, France}
\email{abderrahmane.habbal@univ-cotedazur.fr}

\author{M. Jahid}

\author{L. Maniar}

\author{A. Ratnani}

 
\makeatletter %added for 2020 MSC
\@namedef{subjclassname@2020}{%
  \textup{2020} Mathematics Subject Classification}
\makeatother

\begin{document}
\dedicatory{\large Dedicated to the memory of Professor Hammadi Bouslous}
\begin{abstract}
We investigate inverse backward-in-time problems for a class of second-order degenerate Mean-Field Game (MFG) systems. More precisely, given the final datum $(u(\cdot, T),m(\cdot, T))$ of a solution to the one-dimensional mean-field game system with a degenerate diffusion coefficient, we aim to determine the intermediate states $(u(\cdot,t_{0}),m(\cdot,t_{0}))$ for any $t_{0} \in [0, T)$, i.e., the value function and the mean distribution at intermediate times, respectively. We prove conditional stability estimates under suitable assumptions on the diffusion coefficient and the initial state $(u(\cdot,0),m(\cdot,0))$. The proofs are based on Carleman's estimates with a simple weight function. We first prove a Carleman estimate for the Hamilton-Jacobi-Bellman (HJB) equation. A second Carleman estimate will be derived for the Fokker-Planck (FP) equation. Then, by combining the two estimates, we obtain a Carleman estimate for the mean-field game system, leading to the stability of the backward problems.
\end{abstract}

\maketitle

\section{Introduction and motivation}
In this paper, we are interested in the stability of backward inverse problems for nonlinear second-order MFG systems in the one-dimensional {\it degenerate} case. Degenerate MFG with possibly vanishing diffusion may indeed occur in many situations, where randomness is hindered at some stage, as e.g. for certain robotic or autonomous systems, particularly when robots or drones need to execute highly precise movements (e.g., in assembly lines or surgical robots). When precise control is required -for instance, at the boundary-, the robots' movements are largely deterministic with minimal randomness. The diffusion term representing random fluctuations becomes negligible, and the MFG describing the system dynamics degenerates, reflecting this precise control situation.
We can also find this type of equation in the field of water resources (water management) \cite{GRP}, as well as in the field of oil production \cite{OJP}.

Let $\Omega :=(0,1)$ be the domain where agents are located and let $T>0$ be a terminal time. Denote $Q:=(0,1) \times(0, T)$, $\Sigma:=\left\lbrace0,1\right\rbrace \times(0, T)$, and
consider the following degenerate MFG system with a quadratic Hamiltonian
\begin{empheq}[left = \empheqlbrace]{alignat=2}
\begin{aligned} \label{eq0}
& u_t(x,t) +a(x) u_{xx}(x,t)-\frac{p(x,t) }{2}|u_{x}|^2+   d(x,t) m = F(x,t), &&\quad \text { in } Q,\\
& m_{t}(x,t) -\left(a(x) m(x,t)\right)_{xx} -\left( p(x,t)mu_{x}\right)_{x}= G(x,t), &&\quad \text { in } Q,\\
&  u= am=0, &&\quad \text { on } \Sigma,\\
& m(\cdot,0)=m_0,\quad u(\cdot,T)=h, &&\quad \text { in } \Omega.
\end{aligned} 
\end{empheq}
Throughout the paper, we assume that the diffusion coefficient satisfies: $a\in C([0,1])\cap C ^1(0,1)$, $a(x)>0$ for $x\in \Omega$ and $a(0)=a(1)=0$, i.e., it degenerates at the boundary points $x=0$ and $x=1$. Additional assumptions will be discussed later.

We aim to determine the unknown intermediate state $(u(\cdot,t_{0}),m(\cdot,t_{0}))$, for every $t_{0}\in [0, T)$,  in the system \eqref{eq0} from the knowledge of the final datum $(u(\cdot,T),m(\cdot,T))$. It should be emphasized that the problem for $m(\cdot,t_{0})$ is backward in the opposed time direction (the FP equation is forward), while the problem for $u(\cdot,t_{0})$ goes in the same direction of time (the HJB equation is backward). Still, we call the original problem backward.
\smallskip

The Mean-Field Games theory is devoted to analyzing differential games with infinitely many agents trying to optimize their costs. The rigorous foundations of the mean-field game theory were initially proposed by Lasry and Lions \cite{LPL06, LPL061}, and Huang, Caines, and Malhamé \cite{HCM07, HCM06}. These pioneering works simplified the theoretical framework of games into a system of two PDEs: the HJB equation and the FP equation (see system \eqref{eq0}). The first equation is backward in time, while the latter is forward in time. The HJB equation describes the value function $u$ for a representative player, while the FP equation governs the distribution mean-field dynamics $m$ of all players. The terminal condition $u(\cdot, T)$ for the HJB equation represents the payoff for the representative player at the end of the game, while the initial condition $m(\cdot, 0)$ for the FP equation characterizes the initial distribution of all players. This allows for the solution of the infinite-player game based on a solution of the MFG systems. However, the presence of nonlinearities in the system and the second-order coupling leads to substantial difficulties. MFG theory has been applied to a wide range of models, including those related to economics and finance \cite{Tr21}, pedestrian crowds \cite{Ach20, LT11}, vehicle interactions \cite{Ma23}, public health (vaccination games) \cite{D22}, etc.

As a simplified interpretation of system \eqref{eq0}, we consider infinitely many agents whose movements are influenced by a combination of their control (strategy) $\alpha_t$ and a random noise modeled by a Brownian motion $B_t$. The  following stochastic differential equation describes the dynamics of each agent
\begin{equation}\label{sde}
dX_t= p(X_t,t)\alpha_t dt + \sigma(X_t)dB_t.    
\end{equation}
Each agent seeks to minimize the following cost function 
$$ J\left(\alpha\right)=\mathbb{E}\left[\int_{0}^T \mathcal{L}(X_t,m(X_t,t))+ \frac{p(X_t,t)}{2}\vert \alpha_t\vert^2 dt+ h(X_T)\right],$$
where $X_t:=X(t)$ is the state of an agent at time $t$, $h$ is the final cost, $\alpha_t:=\alpha(t)$ is a control, $p(x,t)$ can be interpreted as the elasticity of the medium, $\mathcal{L}$ is the running cost (the cost of interactions with other agents), and $\mathbb{E}$ denotes the mathematical expectation corresponding to the Brownian motion $B_t$.

It should be emphasized that degenerate mean-field game systems such as \eqref{eq0} find many applications, although not well considered. Indeed, the diffusion coefficient $\sigma$ in \eqref{sde} can vanish at some points, and so can $ a(x)= \frac{1}{2}\sigma^2(x)$. For instance, an interesting example is given by $a(x)=x^\beta (1-x)^\delta, \; \beta,\delta>0$. This model is related to the Wright-Fischer diffusion process in $[0,1]$ used for gene frequency models in population genetics, where $a(x)=x(1-x)$. See \cite[Examples, p.~48]{Ma68} and also \cite[Ch.~VIII, \S 8]{Sh92} for multidimensional models. We refer to \cite{OJP} for a model of oil production with $a(x)=\frac{\gamma^2}{2}x^2$, where $\gamma$ measures the intensity of the noise. A similar model for water management has recently been studied in \cite{GRP}. In the paper \cite{CGP15}, the authors analyzed {\it possibly degenerate} general second-order MFG systems with local coupling. The authors proved the existence and uniqueness of appropriately defined weak solutions.

During the last few years, inverse problems for MFG systems have been actively investigated. Recently in \cite{Kl0}, and for the first time, Klibanov and Averboukh introduced Carleman estimates (weighted energy estimates with large parameters \cite{Car39, Ya}) within the analysis of MFG systems. By developing new Carleman estimates, they have proven a Lipschitz stability estimate concerning errors in the initial and terminal input data. Then several works followed, and the field is rapidly growing. We refer to \cite{Mk} where a quasi-Carleman estimate is introduced to prove stability and uniqueness of an overdetermined MFG system. The authors in \cite{KL23} have proven a Hölder stability estimate for MFG systems with the lateral Cauchy data (Dirichlet and Neumann boundary data). In \cite{ILY23}, a unique continuation problem for MFG systems was considered. The authors of \cite{LY23} have proven Hölder and Lipschitz stability estimates for the states for a linearized MFG system. Furthermore, in \cite{ILY231}, the Lipschitz stability for an inverse source problem using boundary data has recently been proven. In \cite{IY23}, the authors have studied an inverse problem for a Hamiltonian coefficient. The recent paper \cite{KLL23} studied coefficient inverse problems for generalized MFG systems with final overdetermination. Some Hölder stability estimates are proven in the case of partial boundary data, and Lipschitz stability estimates are proven in the case of complete boundary data. These estimates imply, in particular, uniqueness results. Similar results have been proven in \cite{Kl23}. However, all the above papers dealt with the case of uniformly parabolic equations, i.e., the case where the diffusion coefficient is positive up to the boundary. For some related inverse problems of linear degenerate parabolic equations, we refer to \cite{Ka15}.

Our main purpose in this paper is to establish conditional stability estimates for backward-in-time problems in the degenerate case using Carleman estimates. We will establish three Carleman estimates: the first one for the Hamilton-Jacobi-Bellman equation and the second one for the Fokker-Planck equation. By combining the two Carleman estimates, we obtain a Carleman estimate for the MFG system. To our knowledge, there is no work on backward problems for degenerate MFG systems.

The paper is organized as follows. In Section \ref{sec2}, we prove a Carleman estimate for the linearized mean-field game system. In Section \ref{sec3}, we show the conditional stability of the backward problem for the linearized degenerate mean-field game system. In Section \ref{sec4}, we prove the conditional stability for the nonlinear degenerate mean-field game system. Finally, Section \ref{sec5} is devoted to some conclusions and perspectives.
\section{Carleman estimates for linearized MFG system}\label{sec2}
We consider the linearized degenerate mean-field game system
\begin{empheq}[left = \empheqlbrace]{alignat=2}
\begin{aligned}\label{eq1to15}
& u_t(x,t) +a(x) u_{xx}(x,t)+d_{1}(x,t) u_{x}=   d_{2}(x,t) m + F, &&\quad \text { in } Q, \\
& m_{t}(x,t) -\left(a(x) m(x,t)\right)_{xx} +c_{1}(x,t)m_{x}=b(x,t) m+ c_{2}(x,t)u_{x} + \rho(x,t) u_{xx} + G, &&\quad \text { in } Q, \\
& u= am=0, &&\quad \text { on } \Sigma\\
& m(\cdot,0)=m_0,\quad u(\cdot,T)=h, &&\quad \text { in } \Omega.
\end{aligned}
\end{empheq}
We assume the following assumptions: $b,c_{2},\rho \in L^{\infty}(Q)$ and there exists a constant $C>0$ such that for all $(x,t)\in Q$:
\begin{equation}
\vert d_{1}(x,t)\vert \leqslant C \sqrt{a(x)} ,\;\;\;\; \vert c_{1}(x,t)\vert \leqslant C \sqrt{a(x)} ,\;\;\;\; \vert a_x(x)\vert \leqslant C \sqrt{a(x)}, \quad \vert d_{2}(x,t)\vert \leqslant C a(x).
\end{equation}
In this section, we prove a Carleman estimate for the linearized MFG system \eqref{eq1to15}. To do this, we first develop a Carleman estimate for a degenerate Hamilton-Jacobi-Bellman equation. Then we derive a Carleman estimate for a degenerate Fokker-Planck equation. We start by presenting some results for linear degenerate equations which are of independent interest.

\subsection{Linear part of the HJB equation}
\subsubsection{Carleman estimate without drift term }
We consider the following backward degenerate parabolic equation in nondivergence form
\begin{empheq}[left = \empheqlbrace]{alignat=2}
\begin{aligned}\label{eq1to1}
& u_{t}(x,t) +a(x) u_{xx}(x,t)= F(x,t), &&\quad \text { in } Q,\\
& u=0, &&\quad \text{ on } \Sigma,\\
&u(x,T)=u_T, &&\quad \text{ in } \Omega.
\end{aligned}
\end{empheq}
To study \eqref{eq1to1}, let us consider the weighted Hilbert spaces 
$$L^2_{\frac{1}{a}}(0,1):=\left\lbrace u 
 \in L^2(0,1)\; |\;\|u\|_{\frac{1}{a}}<\infty \right\rbrace, \quad \|u\|_{\frac{1}{a}}^2:=\int_{0}^1u^2 \frac{1}{a} dx;
 $$
$$H^1_{\frac{1}{a}}(0,1):=L^2_{\frac{1}{a}}(0,1) \cap H^1_{0}(0,1), \quad \|u\|_{1,\frac{1}{a}}^2:= \|u\|_{\frac{1}{a}}^2+\| u_{x}\|_{L^2(0,1)}^2;$$
$$H^2_{\frac{1}{a}}(0,1):=\left\lbrace u 
 \in H^1_{\frac{1}{a}}(0,1)\; | \; a u_{xx}\in L^2_{\frac{1}{a}}(0,1)\right\rbrace, \quad \|u\|_{2,\frac{1}{a}}^2:= \|u\|_{1,\frac{1}{a}}^2+\int_{0}^1 au_{xx}^2  dx.$$
We adopt the same weight function as in \cite{CY23}:
$$ \varphi(t)=e^{\lambda t}, \;\; t>0,$$
where $\lambda>0$ is a sufficiently large parameter. Note that this simple weight function is $x$-independent and can be used in degenerate and nondegenerate cases.

Next, we state the Carleman estimate relevant to \eqref{eq1to1}.
\begin{lemma}
There exists a constant $\lambda_{0}>0$ such that for each $\lambda>\lambda_{0}$, there are constants $s_{0}(\lambda)>0$ and $C>0$ such that
\begin{equation}\label{Carlm HJB}
    \begin{aligned}
           &\int_{Q} \left(\frac{1}{a(x)}\vert u_{t}\vert^2+ a(x)\vert u_{xx}\vert^2+ s\lambda \varphi \vert u_{x}\vert^2+ \frac{s^2\lambda^2\varphi^2}{a(x)}\vert u\vert^2\right)e^{2s\varphi}dxdt\\&\leqslant C  \int_{Q}s\varphi\frac{1}{a(x)}\vert F\vert^2 e^{2s\varphi}dxdt +Cs\left( s\lambda\varphi(T)\| u(\cdot,T)\|_{\frac{1}{a}}^2+\| u(\cdot,T)\|_{1,\frac{1}{a}}^2 \right)e^{2s\varphi(T)}\\& +Cs\left( s\lambda\| u(\cdot,0)\|_{\frac{1}{a}}^2+\| u(\cdot,0)\|_{1,\frac{1}{a}}^2\right)e^{2s}
    \end{aligned}
\end{equation}
for all $s\geqslant s_{0}(\lambda)$ and $u \in H^2_{\frac{1}{a}}(0,1)$ solution of \eqref{eq1to1}.
\end{lemma}
\begin{proof}
Set 
$$ Lu :=u_{t} + a(x)  u_{xx}, \quad  w=e^{s\varphi}u, \quad Pw=e^{s\varphi}L\left(e^{-s\varphi}w\right)= e^{s\varphi} F.
$$
Hence
   $$ e^{s\varphi} \left(e^{-s\varphi}w\right)_{t}=w_{t} -s\lambda \varphi w, $$
    $$ e^{s\varphi} a(x) \left(w e^{-s\varphi}\right)_{xx}=a(x)w_{xx},$$
    and 
    $$ Pw= e^{s\varphi}L\left(e^{-s\varphi}w\right)= P_{1}w+P_{2}w =e^{s\varphi} F,
    $$
   where 
   $$
   P_{1}w= w_{t}, \quad P_{2}w = -s\lambda \varphi w +a(x)w_{xx}.
   $$
By taking the norm $\|\cdot\|_{L^2((0,T),L^2_{\frac{1}{a}}(0,1))}$ in the previous equality, we obtain
\begin{align*}
\| e^{s\varphi}F\|^2_{L^2((0,T),L^2_{\frac{1}{a}}(0,1))}=&\int_{Q} \frac{1}{a(x)}\vert w_{t}\vert^2 dx dt + 2 \int_{Q} \frac{1}{a(x)}w_{t} \left(-s\lambda \varphi w +a(x)w_{xx} \right)\\
& +\int_{Q}\frac{1}{a(x)}\left\vert - s\lambda \varphi w +a(x)w_{xx} \right\vert^2dxdt\\
 \geqslant& \int_{Q}\frac{1}{a(x)} \vert w_{t}\vert^2dxdt + 2 \int_{Q}w_{t}w_{xx} dx dt+2 \int_{Q}\frac{1}{a(x)}w_{t}  (-s\lambda\varphi)wdxdt\\
:= &\int_{Q} \frac{1}{a(x)}\vert w_{t}\vert^2dxdt+ J_{1}+J_{2}.
\end{align*}
Thus 
$$ \int_{Q}\frac{1}{a(x)}\vert F\vert^2 e^{2s\varphi}dx dt\geqslant J_{1}+J_{2},$$
\begin{align}\label{3}
    \int_{Q}\frac{1}{a(x)} \vert w_{t}\vert^2dxdt \leqslant \int_{Q}\frac{1}{a(x)}\vert F\vert^2 e^{2s\varphi}dx dt -J_{1}-J_{2},
\end{align}
and 
\begin{align*}
    \int_{Q}\frac{1}{a(x)} \vert P_{2}w\vert^2dxdt \leqslant \int_{Q}\frac{1}{a(x)}\vert F\vert^2 e^{2s\varphi}dx dt -J_{1}-J_{2}.
\end{align*}
Take $s>1$ and $\lambda>1$, 
using the boundary conditions and integration by part, we obtain
\begin{equation}\label{4}
    \begin{aligned}
    J_{1}&=2 \int_{Q}w_{t} w_{xx} dx dt= -2\int_{Q}w_{xt}w_{x}dx dt=-\int_{Q}(w_{x}^2)_{t}dx dt\\
&= \int_{0}^1\left[\vert w_{x}(0,x)\vert^2- \vert w_{x}(T,x)\vert^2\right] dx\leqslant  \int_{0}^1\left[\vert w_{x}(T,x)\vert^2+\vert w_{x}(0,x)\vert^2\right] dx.
\end{aligned}
\end{equation}
On the other hand,
\begin{equation}\label{5}
    \begin{aligned}
    J_{2} &=-s\lambda \int_{Q} \frac{1}{a(x)}2w_{t} w\varphi\, dxdt= -s\lambda \int_{Q} \frac{1}{a(x)}\left(w^2\right)_{t}\varphi\, dxdt\\
     & =s\lambda \int_{Q} \frac{1}{a(x)}\varphi'(t)w^2\, dxdt-s\lambda \int_{0}^1\left[\frac{1}{a(x)}\varphi w^2\right]^{t=T}_{t=0} dx\\
     &= s\lambda^2 \int_{Q}\varphi \frac{1}{a(x)} w^2 \, dx dt-s\lambda \int_{0}^1\frac{1}{a(x)}\left(\varphi(T)\vert w(x,T)\vert^2-\vert w(x,0)\vert^2\right)dx.
\end{aligned}
\end{equation}
Then 
\begin{equation}\label{1}
\begin{aligned}
    \| e^{s\varphi}F\|^2_{L^2((0,T),L^2_{\frac{1}{a}}(0,1))} &\geqslant s\lambda^2 \int_{Q}\varphi\frac{1}{a(x)} w^2 \,dxdt  -s\lambda \int_{0}^1 \frac{1}{a(x)}\left(\varphi(T)\vert w(x,T)\vert^2+\vert w(x,0)\vert^2\right)dx\\& -\int_{0}^1\left( \vert w_{x}(x,T)\vert^2+\vert w_{x}(x,0)\vert^2\right)\, dx.
\end{aligned}
\end{equation}
Moreover,
\begin{align*}
    \int_{Q}\frac{1}{a(x)}\left(Pw\right)(-w)\, dxdt&=- \int_{Q} \frac{1}{a(x)}w_{t}w\, dx dt + \int_{Q}s\lambda \varphi\frac{1}{a(x)} w^2\, dx dt- \int_{Q} \frac{1}{a(x)} \left( a(x)w_{xx}\right)w \,dxdt\\
    &= -\int_{Q} \frac{1}{a(x)}w_{t}w\, dxdt +\int_{Q}s\lambda \varphi\frac{1}{a(x)} w^2\, dxdt- \int_{Q} w_{xx}w \,dxdt\\
    &  =I_{1}+I_{2}+I_{3}.
\end{align*}
\begin{align*}
   \left\vert I_{1}\right\vert &= \left\vert - \int_{Q}\frac{1}{a(x)}w_{t}w\, dxdt\right\vert = \left\vert -\frac{1}{2} \int_{Q}\frac{1}{a(x)}(w^2)_{t} \, dxdt\right\vert\\ & =\frac{1}{2}\left\vert -\int_{0}^1\frac{1}{a(x)}\left[ \vert w(x,t)\vert^2\right]_{t=0}^{t=T}\, dx\right\vert \leq \frac{1}{2}\int_{0}^1\frac{1}{a(x)}\left(\vert w(x,T)\vert^2+\vert w(x,0)\vert^2 \right)\, dx.
\end{align*}
Next,
$$ \vert I_{2}\vert = \left\vert \int_{Q} s\lambda \varphi \frac{1}{a(x)} w^2\, dxdt\right\vert \leq  \int_{Q} s\lambda \varphi \frac{1}{a(x)} w^2 \, dxdt,$$
and 
\begin{align*}
     I_{3}&= -\int_{Q}w_{xx}w dxdt= \int_{Q}\vert w_{x}\vert^2  dxdt.
\end{align*}
Hence
\begin{align*}
    \int_{Q}\lambda \frac{1}{a(x)}\left(Pw \right)(-w) dxdt &\geqslant   \int_{Q}\lambda \vert w_{x}\vert^2 dx dt- \int_{Q} s\lambda^2 \varphi \frac{1}{a(x)} w^2 dxdt\\& - \int_{0}^1\frac{\lambda}{2a(x)}\left(\vert w(x,T)\vert^2+\vert w(x,0)\vert^2 \right)\, dx.
\end{align*}
Moreover,
\begin{align*}
\left\vert \int_{Q} \frac{1}{a(x)} \lambda(Pw)(-w) dxdt\right\vert  & \leqslant \frac{1}{2}\int_{Q}\frac{1}{a(x)}\vert Pw\vert^2 dx dx dt+\frac{\lambda^2}{2}\int_{Q}\frac{1}{a(x)}\vert w\vert^2 dx dt\\ & =\frac{1}{2} \int_{Q} \frac{1}{a(x)}\vert F \vert^2 e^{2s\varphi} dxdt +\frac{\lambda^2}{2}\int_{Q}\frac{1}{a(x)}\vert w\vert^2 dx dt,
\end{align*}
and
\begin{equation*}\label{101}
   \begin{aligned}
    \lambda \int_{Q} \vert w_{x}\vert^2  dx dt \leqslant &\int_{Q}s\lambda^2\varphi \frac{1}{a(x)}w^2 dxdt  + \frac{1}{2} \int_{Q} \frac{1}{a(x)}\vert F \vert^2 e^{2s\varphi} dxdt \\&+\frac{\lambda^2}{2}\int_{Q}\frac{1}{a(x)}\vert w\vert^2 dx dt + \frac{\lambda}{2}\int_{0}^1\frac{1}{a(x)}\left(\vert w(x,T)\vert^2+\vert w(x,0)\vert^2 \right)\, dx.
\end{aligned} 
\end{equation*}
Using \eqref{1} to estimate the first term on the right-hand side, we have
\begin{equation}\label{2}
    \begin{aligned}
 & \lambda \int_{Q} \vert w_{x}\vert^2  dx dt\leqslant C_{3}\int_{Q} \frac{1}{a(x)}\vert F e^{s\varphi}\vert^2 dxdt +C_{3}\int_{Q}\lambda^2\frac{1}{a(x)} \vert w\vert^2 dxdt+C_{3}\lambda\left(\| w(\cdot,T)\|_{\frac{1}{a}}^2+\| w(\cdot,0)\|_{\frac{1}{a}}^2 \right)  \\ & +C_{3}s\lambda \left(\varphi(T)\| w(\cdot,T)\|_{\frac{1}{a}}^2+\| w(\cdot,0)\|_{\frac{1}{a}}^2 \right) +C_{3}\left( \| w_{x}(\cdot,T)\|_{L^2(0,1)}^2+\| w_{x}(\cdot,0)\|_{L^2(0,1)}^2\right)\\
 & \leqslant C_{3}\int_{Q}\frac{1}{a(x)}\vert F e^{s\varphi}\vert^2 dxdt + C_{3}\int_{Q}\lambda^2 \frac{1}{a(x)} w^2 dxdt +C_{3}s\lambda \left(\varphi(T)\| w(\cdot,T)\|_{\frac{1}{a}}^2+\| w(\cdot,0)\|_{\frac{1}{a}}^2 \right) \\ &+C_{3}\left( \| w_{x}(\cdot,T)\|_{L^2(0,1)}^2+\| w_{x}(\cdot,0)\|_{L^2(0,1)}^2\right).
\end{aligned}
\end{equation}
Adding \eqref{1} and \eqref{2}, we have 
\begin{align*}
    &\int_{Q} s\lambda^2 \varphi \frac{1}{a(x)} w^2 dxdt+\lambda \int_{Q} \vert w_{x}\vert^2  dx dt  \leqslant C_{4}  \int_{Q} \frac{1}{a(x)}\vert F e^{s\varphi}\vert^2 dxdt +C_{4}\int_{Q}\lambda^2  \frac{1}{a(x)} w^2 dxdt\\
    &+C_{4}s\lambda \left(\varphi(T)\| w(\cdot,T)\|_{\frac{1}{a}}^2+\| w(\cdot,0)\|_{\frac{1}{a}}^2 \right) +C_{4}\left( \| w_{x}(\cdot,T)\|_{L^2(0,1)}^2+\| w_{x}(\cdot,0)\|_{L^2(0,1)}^2\right).
\end{align*}
Given that $\varphi(t)=e^{\lambda t}\geqslant 1$, we choose a sufficiently large $s>0$ to absorb the second term on the right-hand side by the first term of the left-hand side. Hence,
\begin{equation}\label{6}
\begin{aligned}
    &\int_{Q} s\lambda^2 \varphi \frac{1}{a(x)}w^2 dxdt+\lambda \int_{Q}\vert w_{x}\vert^2  dx dt \leqslant C_{4}  \int_{Q}\frac{1}{a(x)}\vert F e^{s\varphi}\vert^2 dxdt \\&+C_{4}s\lambda \left(\varphi(T)\| w(\cdot,T)\|_{\frac{1}{a}}^2+\| w(\cdot,0)\|_{\frac{1}{a}}^2 \right) +C_{4}\left( \| w_{x}(\cdot,T)\|_{L^2(0,1)}^2+\| w_{x}(\cdot,0)\|_{L^2(0,1)}^2\right).
\end{aligned}
\end{equation}
We now focus on estimating $\vert w_{t}\vert^2$. Since we have $u=e^{-s\varphi}w,$ then
$$ u_{t}=-s\lambda \varphi e^{-s\varphi}w+e^{-s\varphi}w_{t}$$ and 
$$ \frac{1}{s\varphi a(x)}\vert u_{t}\vert^2e^{2s\varphi }\leqslant 2s\lambda^2 \varphi\frac{1}{a(x)} w^2+\frac{2}{s\varphi a(x)}\vert w_{t}\vert^2. $$
Let $\epsilon \in \left(0,\frac{1}{2}\right)$ be a parameter that we choose later. We note that $\frac{1}{s\varphi}= \frac{1}{se^{\lambda t}}\leq \frac{1}{s}\leq \frac{1}{2}$ for $s \geqslant 2.$ Hence, for all large $s>0$ and $\lambda> 0,$ it follows by \eqref{3} that
\begin{equation*}
    \begin{aligned}
        &\int_{Q}\frac{\epsilon}{s\varphi a(x)}\vert u_{t}\vert^2 e^{2s\varphi}\, dxdt\leqslant \int_{Q}2\epsilon s \lambda^2 \varphi \frac{1}{a(x)}w^2dxdt+\int_{Q}\frac{2\epsilon}{s\varphi a(x)}\vert w_{t}\vert^2dxdt\\ &\leqslant 2\epsilon \int_{Q}s\lambda^2\varphi \frac{1}{a(x)} w^2 dxdt+\epsilon \int_{Q}\frac{1}{a(x)}\vert w_{t}\vert^2dxdt\\ & \leqslant 2\epsilon \int_{Q} s\lambda^2\varphi \frac{1}{a(x)} w^2 dxdt + \epsilon \int_{Q} \frac{1}{a(x)} \vert F\vert^2 e^{2s\varphi}dxdt + \epsilon \left(-J_{1}-J_{2}\right).
    \end{aligned}
\end{equation*}
By applying \eqref{4} and \eqref{5}, we obtain
\begin{align*}
    \epsilon \left(-J_{1}-J_{2}\right)=& -s\epsilon\lambda^2 \int_{Q}\varphi \frac{1}{a(x)} w^2 \, dx dt +s\lambda\epsilon\int_{\Omega} \frac{1}{a(x)}\left(\varphi(T)\vert w(x,T)\vert^2-\vert w(x,0)\vert^2\right)dx\\
   &\,-\epsilon  \int_{0}^1\left(\vert w_{x}(x,0)\vert^2-\vert w_{x}(x,T)\vert^2\right)dx.
\end{align*}
Substituting this in \eqref{6} yields
\begin{equation}\label{7}
    \begin{aligned}
        &\int_{Q}\frac{\epsilon}{s\varphi a(x)}\vert u_{t}\vert^2 e^{2s\varphi}\, dxdt\leqslant \epsilon \int_{Q} s\lambda^2\varphi\frac{1}{a(x)} w^2 dxdt + \epsilon \int_{Q} \frac{1}{a(x)}\vert F\vert^2 e^{2s\varphi}dxdt \\ & +\int_{0}^1\epsilon\left(\vert  w_{x}(x,T)\vert^2+ \vert   w_{x}(x,0)\vert^2\right)dx +\epsilon s\lambda \int_{0}^1\frac{1}{a(x)}\left(\varphi(T)\vert w(x,T)\vert^2+\vert w(x,0)\vert^2\right)dx.
    \end{aligned}
\end{equation}
Adding \eqref{6} to \eqref{7} gives
\begin{equation}\label{8}
    \begin{aligned}
         &\int_{Q} s\lambda^2 \varphi \frac{1}{a(x)}w^2 dxdt+ \int_{Q} \lambda \vert w_{x}\vert^2 dx dt+\int_{Q}\frac{\epsilon}{s\varphi a(x)}\vert u_{t}\vert^2 e^{2s\varphi}\, dxdt\\& \leqslant C_{5}  \int_{Q}\frac{1}{a(x)}\vert F\vert^2 e^{2s\varphi}dxdt + 2\epsilon \int_{Q} s\lambda^2\varphi \frac{1}{a(x)}w^2 dxdt +C_{5}s\lambda \left(\varphi(T)\| w(\cdot,T)\|_{\frac{1}{a}}^2+\| w(\cdot,0)\|_{\frac{1}{a}}^2 \right)\\& +C_{5}\left( \| w_{x}(\cdot,T)\|_{L^2(0,1)}^2+\| w_{x}(\cdot,0)\|_{L^2(0,1)}^2\right).
    \end{aligned}
\end{equation}
Next,  we will estimate $ a(x) w_{xx}^2$. As,  $a(x) w_{xx}=P_{2}w+s\lambda \varphi w$,
\begin{align*}
    \frac{1}{s\varphi}a(x)\vert w_{xx}\vert^2 & \leqslant 2 \frac{1}{s\varphi a(x)}\vert P_{2}w\vert^2 +  2\frac{s \lambda^2 \varphi}{a(x)}\vert w\vert^2 \\ &
   \leqslant \frac{1}{a(x)} \vert P_{2}w \vert^2  + 2\frac{s \lambda^2 \varphi}{a(x)}\vert w\vert^2 .
    \end{align*}
     Therefore, for all large $s>0$ and $\lambda> 0, $ we have
    \begin{equation}\label{9}
        \begin{aligned}
            \int_{Q} \epsilon\frac{1}{s\varphi}a(x)\vert w_{xx}\vert^2 dx dt &\leqslant \int_{Q}\epsilon \frac{1}{a(x)} \vert F\vert^2 e^{2s\varphi}dxdt -\epsilon(J_{1}+J_{2}) + \int_{Q}2\epsilon  \frac{s \lambda^2 \varphi}{a(x)}\vert w\vert^2 dx dt\\&
            \leq \int_{Q}\epsilon \frac{1}{a(x)} \vert F\vert^2 e^{2s\varphi}dxdt +C\int_{0}^1\epsilon\left(\vert  w_{x}(x,T)\vert^2+ \vert   w_{x}(x,0)\vert^2\right)dx \\ &+\int_{Q}\epsilon  \frac{s \lambda^2 \varphi}{a(x)}\vert w\vert^2 dx dt +\epsilon s\lambda \int_{0}^1\frac{1}{a(x)}\left(\varphi(T)\vert w(x,T)\vert^2+\vert w(x,0)\vert^2\right)dx.
        \end{aligned}
    \end{equation}
Adding \eqref{9} to \eqref{8}, we obtain 
\begin{equation*}
    \begin{aligned}
          &\int_{Q} s\lambda^2 \varphi \frac{1}{a(x)}w^2 dxdt+ \int_{Q} \epsilon\frac{1}{s\varphi}a(x)\vert w_{xx}\vert^2 dx dt+\int_{Q} \lambda \vert w_{x}\vert^2 dx dt+\int_{Q}\frac{\epsilon}{s\varphi a(x)}\vert u_{t}\vert^2 e^{2s\varphi}\, dxdt\\& \leqslant C_{5}  \int_{Q}\frac{1}{a(x)}\vert F e^{s\varphi}\vert^2 dxdt + 4\epsilon \int_{Q} s\lambda^2\varphi \frac{1}{a(x)}w^2 dxdt +C_{5}s\lambda \left(\varphi(T)\| w(\cdot,T)\|_{\frac{1}{a}}^2+\| w(\cdot,0)\|_{\frac{1}{a}}^2 \right)\\& +C_{5}\left( \| w_{x}(\cdot,T)\|_{L^2(0,1)}^2+\| w_{x}(\cdot,0)\|_{L^2(0,1)}^2\right).
    \end{aligned}
\end{equation*}
Setting $v=\varphi^{\frac{1}{2}}u$, we obtain $ v_{t}=\frac{1}{2}\lambda \varphi^{\frac{1}{2}}u+\varphi^{\frac{1}{2}}u_{t}$ and
\begin{align*}
    v_{t}+a(x)v_{xx}=\varphi^{\frac{1}{2}}(u_{t}+a(x)u_{xx})+ \frac{1}{2}\lambda\varphi^{\frac{1}{2}} u= \varphi^{\frac{1}{2}} F +\frac{1}{2}\lambda\varphi^{\frac{1}{2}} u.
\end{align*}
Hence, we have 
\begin{equation}\label{14}
    \begin{aligned}
         & \int_{Q} \frac{1}{\varphi}\left(a(x)\vert v_{xx}\vert^2+\frac{1}{ a(x)}\vert v_{t}\vert^2 \right)e^{2s\varphi}dx dt+\int_{Q} \lambda s\vert v_{x}\vert^2 e^{2s\varphi} dx dt+\int_{Q} s^2\lambda^2 \varphi \frac{1}{a(x)}v^2 e^{2s\varphi} dxdt\\& \leqslant C  \int_{Q}s\varphi\frac{1}{a(x)}\vert F\vert^2e^{2s\varphi} dxdt +\frac{C}{2}\int_{Q}\frac{s\lambda^2\varphi}{a(x)}u^2 e^{2s\varphi}dxdt +C s\left( s\lambda\varphi(T)\| v(\cdot,T)\|_{\frac{1}{a}}^2+\| v(\cdot,T)\|_{1,\frac{1}{a}}^2 \right)e^{2s\varphi(T)}\\& 
         +Cs\left( s\lambda\| v(\cdot,0)\|_{\frac{1}{a}}^2+\| v(\cdot,0)\|_{1,\frac{1}{a}}^2\right)e^{2s}.
    \end{aligned}
\end{equation}
On the other hand,
\begin{equation*}
    \begin{aligned}
        &\varphi\vert u_{t}\vert^2 \leqslant 2\vert v_{t}\vert^2 +\lambda^2\varphi\vert u\vert^2,\qquad 
        \varphi \vert u_{xx}\vert^2 = \vert v_{xx}\vert^2.
    \end{aligned}
\end{equation*}
Hence, 
\begin{align*}
    &\frac{1}{a(x)}\vert u_{t}\vert^2+ a(x)\vert u_{xx}\vert^2\\&
    \leqslant \frac{1}{\varphi}\left( \frac{2}{a(x)}\vert v_{t}\vert^2 +a(x)\vert v_{xx}\vert^2\right)+ \frac{\lambda^2}{a(x)}\vert u \vert^2.
\end{align*}
Therefore,
\begin{align*}
    &\frac{1}{a(x)}\vert u_{t}\vert^2+ a(x)\vert u_{xx}\vert^2+ s\lambda \varphi \vert u_{x}\vert^2+ \frac{s^2\lambda^2\varphi^2}{a(x)}\vert u\vert^2\\&
    \leqslant \frac{1}{\varphi}\left( \frac{2}{a(x)}\vert v_{t}\vert^2 +a(x)\vert v_{xx}\vert^2\right)+ \frac{\lambda^2}{a(x)}\vert u \vert^2 + s\lambda\varphi \vert u_{x}\vert^2+ \frac{s^2\lambda^2\varphi^2}{a(x)}\vert u\vert^2 \\ &\leqslant C\frac{1}{\varphi}\left( \frac{1}{a(x)}\vert v_{t}\vert^2 +a(x)\vert v_{xx}\vert^2\right)+ s\lambda \vert v_{x}\vert^2+ \frac{s^2\lambda^2\varphi}{a(x)}\vert v\vert^2+ \frac{\lambda^2}{a(x)}\vert u \vert^2 . 
\end{align*}
Applying \eqref{14}, we obtain
\begin{equation*}
    \begin{aligned}
        &\int_{Q} \left(\frac{1}{a(x)}\vert u_{t}\vert^2+ a(x)\vert u_{xx}\vert^2+ s\lambda \varphi \vert u_{x}\vert^2+ \frac{s^2\lambda^2\varphi^2}{a(x)}\vert u\vert^2\right)e^{2s\varphi}dxdt
        \\&
        \leqslant C\int_{Q}\frac{1}{\varphi}\left( \frac{1}{a(x)}\vert v_{t}\vert^2 +a(x)\vert v_{xx}\vert^2\right)e^{2s\varphi}+\int_{Q}s\lambda \vert v_{x}\vert^2e^{2s\varphi}+\int_{Q} \frac{s^2\lambda^2\varphi}{a(x)}\vert v\vert^2 e^{2s\varphi}+ \int_{Q} \frac{\lambda^2}{a(x)}\vert u \vert^2 e^{2s\varphi} \\&\leqslant C  \int_{Q}s\varphi\frac{1}{a(x)}\vert F\vert^2 e^{2s\varphi}+C\frac{1}{2}\int_{Q}\frac{s\lambda^2\varphi}{a(x)}u^2 e^{2s\varphi} +C s\left( s\lambda\varphi(T)\| v(\cdot,T)\|_{\frac{1}{a}}^2+\| v(\cdot,T)\|_{1,\frac{1}{a}}^2 \right)e^{2s\varphi(T)}\\& +Cs\left( s\lambda\| v(\cdot,0)\|_{\frac{1}{a}}^2+\| v(\cdot,0)\|_{1,\frac{1}{a}}^2\right)e^{2s}+\int_{Q} \frac{\lambda^2}{a(x)}\vert u \vert^2 e^{2s\varphi} dxdt.
    \end{aligned}
\end{equation*}
For sufficiently large $s>0$, we can absorb the second term and the last term to the left-hand side, and we obtain
\begin{align*}
    &\int_{Q} \left(\frac{1}{a(x)}\vert u_{t}\vert^2+ a(x)\vert u_{xx}\vert^2+ s\lambda \varphi \vert u_{x}\vert^2+ \frac{s^2\lambda^2\varphi^2}{a(x)}\vert u\vert^2\right)e^{2s\varphi}dxdt\\&\leqslant C  \int_{Q}s\varphi\frac{1}{a(x)}\vert F\vert^2 e^{2s\varphi}dxdt +C s\left( s\lambda\varphi(T)\| v(\cdot,T)\|_{\frac{1}{a}}^2+\| v(\cdot,T)\|_{1,\frac{1}{a}}^2 \right)e^{2s\varphi(T)}\\& \quad +Cs\left( s\lambda\| v(\cdot,0)\|_{\frac{1}{a}}^2+\| v(\cdot,0)\|_{1,\frac{1}{a}}^2\right)e^{2s}.
\end{align*}
This achieves the proof. \end{proof}
\subsubsection{Carleman estimate with a drift term}
Let us consider the degenerate parabolic problem with a drift term
\begin{empheq}[left = \empheqlbrace]{alignat=2}
\begin{aligned}
& u_{t}(x,t) +a(x) u_{xx}(x,t)+b(x,t)u_x= F(x,t), &&\quad \text { in } Q,\\
& u=0, &&\quad \text{ on } \Sigma,\\
&u(\cdot,T)=u_T, &&\quad \text{ in } \Omega,
\label{eq1driftterm}
\end{aligned}
\end{empheq}
where $b\in L^{\infty}(Q)$. We assume the following assumption
\begin{align}\label{hypothese}
    \vert b(x,t)\vert \leqslant C \sqrt{a(x)}, \qquad (x,t) \in Q
\end{align}
for some constant $C>0$. See \cite{CY23} for a similar assumption in the divergence form case.
\begin{lemma}
There exists a constant $\lambda_{0}>0$ such that for each $\lambda>\lambda_{0}$ we can find a constant $s_{0}(\lambda)>0$ satisfying: there exists a constant $C>0$ such that
\begin{equation}\label{driftterm}
    \begin{aligned}
           &\int_{Q} \left(\frac{1}{a(x)}\vert u_{t}\vert^2+ a(x)\vert u_{xx}\vert^2+ s\lambda \varphi \vert u_{x}\vert^2+ \frac{s^2\lambda^2\varphi^2}{a(x)}\vert u\vert^2\right)e^{2s\varphi}dxdt\\&\leqslant C  \int_{Q}s\varphi\frac{1}{a(x)}\vert F\vert^2 e^{2s\varphi}dxdt +Cs\left( s\lambda\varphi(T)\| u(\cdot,T)\|_{\frac{1}{a}}^2+\| u(\cdot,T)\|_{1,\frac{1}{a}}^2 \right)e^{2s\varphi(T)}\\& +Cs\left( s\lambda\| u(\cdot,0)\|_{\frac{1}{a}}^2+\| u(\cdot,0)\|_{1,\frac{1}{a}}^2\right)e^{2s}
    \end{aligned}
\end{equation}
for all $s\geqslant s_{0}(\lambda)$ and $u \in H^2_{\frac{1}{a}}(0,1)$ solution of \eqref{eq1driftterm}.
\end{lemma}
\begin{proof}
Set $\Tilde{F}=F-b(x,t)u_x$. As a consequence of Lemma \ref{Carlm HJB}, we obtain
\begin{equation*}
    \begin{aligned}
           &\int_{Q} \left(\frac{1}{a(x)}\vert u_{t}\vert^2+ a(x)\vert u_{xx}\vert^2+ s\lambda \varphi \vert u_{x}\vert^2+ \frac{s^2\lambda^2\varphi^2}{a(x)}\vert u\vert^2\right)e^{2s\varphi}dxdt\\&\leqslant C  \int_{Q}s\varphi\frac{1}{a(x)}\vert \Tilde{F}\vert^2e^{2s\varphi} dxdt +Cs\left( s\lambda\varphi(T)\| u(\cdot,T)\|_{\frac{1}{a}}^2+\| u(\cdot,T)\|_{1,\frac{1}{a}}^2 \right)e^{2s\varphi(T)}\\& +Cs\left( s\lambda\| u(\cdot,0)\|_{\frac{1}{a}}^2+\| u(\cdot,0)\|_{1,\frac{1}{a}}^2\right)e^{2s}\\&\leqslant C  \int_{Q}s\varphi\frac{1}{a(x)}\vert F-b(x,t)u_{x}\vert^2e^{2s\varphi} dxdt +Cs\left( s\lambda\varphi(T)\| u(\cdot,T)\|_{\frac{1}{a}}^2+\| u(\cdot,T)\|_{1,\frac{1}{a}}^2 \right)e^{2s\varphi(T)}\\& +Cs\left( s\lambda\| u(\cdot,0)\|_{\frac{1}{a}}^2+\| u(\cdot,0)\|_{1,\frac{1}{a}}^2\right)e^{2s}\\&\leqslant C  \int_{Q}s\varphi\frac{1}{a(x)}\vert F\vert^2 e^{2s\varphi}dxdt +\int_{Q}s\varphi\frac{b^2(x,t)}{a(x)}\vert u_x \vert^2e^{2s\varphi} dxdt+Cs\left( s\lambda\varphi(T)\| u(\cdot,T)\|_{\frac{1}{a}}^2+\| u(\cdot,T)\|_{1,\frac{1}{a}}^2 \right)e^{2s\varphi(T)}\\& +Cs\left( s\lambda\| u(\cdot,0)\|_{\frac{1}{a}}^2+\| u(\cdot,0)\|_{1,\frac{1}{a}}^2\right)e^{2s},
    \end{aligned}
\end{equation*}
where we have used \eqref{hypothese} and chosen $\lambda >0$ large enough to absorb the second term on the right-hand side.
\end{proof}
\subsection{Linear part of the FP equation} In this section, we will prove a Carleman estimate for the linear part of the Fokker-Planck equation. Note that the literature for such equations is scarce in the degenerate framework. We are only aware of \cite[Example 3.8]{FY98}, where such an equation is briefly discussed on the unweighted space $L^2(\Omega)$.
\subsubsection{Well-posedness}
Let us consider the degenerate parabolic equation
\begin{empheq}[left = \empheqlbrace]{alignat=2}
\begin{aligned}
& m_{t}(x,t) -(a(x) m(x,t))_{xx}= G(x,t), && \quad \text{ in } Q,\\
& am=0, &&\quad \text { on } \Sigma,\\
&m(\cdot,0)=m_0, &&\quad \text{ in } \Omega.
\label{eq1to42}
\end{aligned}
\end{empheq}
We start by studying the well-posedness of \eqref{eq1to42}. Let us consider the following real Hilbert spaces with corresponding norms:
$$L^2_{a}(0,1):= \left\lbrace u \; \text{measurable function}\;|\; \| u\|_{a}< \infty \right\rbrace, \quad \| u\|_{a}^2:=\int_{0}^{1}au^2 dx;$$
$$ H_{1,a}(0,1):=\left\lbrace u \in L^2_{a}(0,1)\;|\; \int_{0}^1\vert (au)_{x}\vert^2 dx< \infty ,\;\; (au)(0)=(au)(1)=0\right\rbrace,$$
$$ \| u\|_{1,a}^2:=\| u\|_{a}^2+\int_{0}^1\vert (au)_{x}\vert^2 dx;$$
$$ H_{2,a}(0,1):=\left\lbrace u \in H_{1,a}(0,1)\;|\; (au)_{xx}\in L^2_{a}(0,1) \right\rbrace,$$
$$\| u\|_{2,a}^2:=\| u\|_{1,a}^2+\int_{0}^1 a\vert (au)_{xx}\vert^2 dx.$$
Then we show the following generation result.
\begin{theorem}
The operator $\left(\Tilde{A},D(\Tilde{A})\right)$ given by 
    $$\Tilde{A} u:=(au)_{xx},\ \qquad D(\Tilde{A})=H_{2,a}(0,1),$$
generates an analytic $C_0$-semigroup on $L^2_a(0,1)$.
\end{theorem}
\begin{proof}
We recall the following operator
$$A u=au_{xx},\qquad D(A)=H_{\frac{1}{a}}^2(0,1).$$
The operator $A$ generates an analytic $C_0$-semigroup $(S(t))_{t\geqslant 0}$ on $L^2_{\frac{1}{a}}(0,1)$ (Theorem 2.3 in \cite{CFR08}). Furthermore, $T:u \mapsto a u$ is an isomorphism from $D(\Tilde{A})=H_{2,a}(0,1) $ onto $D(A)=H_{\frac{1}{a}}^2(0,1),$
and $(\Tilde{S}(t))_{t\geqslant 0}$ defined by $\Tilde{S}(t):=T^{-1}S(t)T$ is an analytic $C_0$-semigroup on $L^2_{a}(0,1) $. Its infinitesimal generator is $T^{-1}AT $ (see \cite{AA}). On the other hand,
\begin{align*}
    T^{-1}A Tu &= T^{-1}A(au)\\&
    =T^{-1}(a(au)_{xx})=(au)_{xx}=\Tilde{A}u.
\end{align*}
Finally, $\Tilde{A}$ is the infinitesimal generator of the analytic $C_0$-semigroup $(\Tilde{S}(t))_{t\geqslant 0}$.
\end{proof}
\subsubsection{Carleman estimate without drift term}
In this section, we prove the following estimate using the Carleman estimate from \cite[Lemma 2.4]{PAPER}.
\begin{lemma}
There exists a constant $\lambda_{0}>0$ such that for each $\lambda>\lambda_{0},$ we can find a constant $s_{0}(\lambda)>0$ so that there exists a constant $C>0$ such that
\begin{equation*}
    \begin{aligned}
         & \int_{Q} \frac{1}{s\varphi}\left(a(x)\vert (am)_{xx}\vert^2+a(x)\vert m_{t}\vert^2 \right)e^{2s\varphi}dx dt+\int_{Q} \lambda \vert (am)_{x}\vert^2 e^{2s\varphi} dx dt+\int_{Q} s\lambda^2 \varphi a(x)\vert m\vert^2 e^{2s\varphi} dxdt\\\leqslant &C \int_{Q}a(x)\vert G\vert^2 e^{2s\varphi}dxdt +C \left(s\lambda\varphi(T)\| m(\cdot,T)\|_{a}^2+\| (am)_{x}(\cdot,T)\|_{2}^2\right)e^{2s\varphi(T)}\\& +C\left( s\lambda\| m(\cdot,0)\|_{a}^2 +\| (am)_{x}(\cdot,0)\|_{2}^2\right)e^{2s}
    \end{aligned}
\end{equation*}
for all $s\geqslant s_{0}(\lambda)$ and $m \in H_{2,a}(0,1)$ satisfying \eqref{eq1to42}.
\end{lemma}
\begin{proof}
Let us consider the change of variable $v=am$, then the equation \eqref{eq1to42} is equivalent to the following one 
\begin{empheq}[left = \empheqlbrace]{alignat=2}
\begin{aligned}
& v_{t}(x,t) - a(x) v_{xx}(x,t)= a(x)G=:F, &&\quad \text{ in } Q,\\
& v=0, &&\quad \text { on } \Sigma.
\label{eq1to43}
\end{aligned}
\end{empheq}
We obtain the desired result by applying \cite[Lemma 2.4]{PAPER}.
\end{proof}
\begin{remark}
Usually, we use the change of variable $ t\mapsto T-t$ to derive a Carleman estimate for a forward equation from a backward equation. However, in our case, we cannot obtain a Carleman estimate for equation \eqref{eq1to43} directly from the Carleman estimate of the backward degenerate equation \eqref{eq1to1} because $\varphi(t)\neq \varphi(T-t)$.
\end{remark}
\subsubsection{Carleman estimate  with a drift term}
Let us consider the degenerate parabolic problem
\begin{empheq}[left = \empheqlbrace]{alignat=2}
\begin{aligned}
& m_{t}(x,t) -( a(x) m(x,t))_{xx} + b(x,t)m_{x}= G(x,t), && \quad \text{ in } Q, \\
& am=0, &&\quad \text { on } \Sigma,
\label{eq1to3}
\end{aligned}
\end{empheq}
where $b\in L^{\infty}(Q)$. Henceforth, we assume the following assumptions
$$ \vert b(x,t)\vert \leqslant C \sqrt{a(x)} ,\qquad \vert a_x(x)\vert \leqslant C \sqrt{a(x)}.$$
Note that imposing the second assumption on $a_x(x)$ for $a(x)=x^\beta(1-x)^\beta$ restricts the exponents to $\beta\geqslant 2.$
\begin{lemma}
    There exists a constant $\lambda_{0}>0$ such that for each $\lambda>\lambda_{0},$ we can find a constant $s_{0}(\lambda)>0$ so that there exists a constant $C>0$ such that
\begin{equation}\label{Carlm FP}
    \begin{aligned}
         & \int_{Q} \frac{1}{s\varphi}\left(a(x)\vert (am)_{xx}\vert^2+a(x)\vert m_{t}\vert^2 \right)e^{2s\varphi}dx dt+\int_{Q} \lambda \vert (am)_{x}\vert^2 e^{2s\varphi} dx dt+\int_{Q} s\lambda^2 \varphi a(x)\vert m\vert^2 e^{2s\varphi} dxdt\\& \leqslant C \int_{Q}a(x)\vert G\vert^2 e^{2s\varphi}dxdt +Cs\lambda \left(\varphi(T)\| m(\cdot,T)\|_{a}^2+\| (am)_{x}(\cdot,T)\|_{2}^2\right)e^{2s\varphi(T)}\\& +C\left( s\lambda\| m(\cdot,0)\|_{a}^2 +\| (am)_{x}(\cdot,0)\|_{2}^2\right)e^{2s}
    \end{aligned}
\end{equation}
for all $s\geqslant s_{0}(\lambda)$ and $m\in H_{2,a}(0,1)$ satisfying \eqref{eq1to3}.
\end{lemma}
\begin{proof} 
Let us consider the change of variable $v=am$, then the equation \eqref{eq1to3} is equivalent to the following one 
\begin{empheq}[left = \empheqlbrace]{alignat=2}
\begin{aligned}
& v_{t}(x,t) - a(x) v_{xx}(x,t)+b(x,t)v_{x}-\frac{b(x,t)a_x(x)}{a(x)}v= a(x)G =F, &&\text{ in } Q,\\
& v=0, &&\text{ on } \Sigma.
\label{eq1to5}
\end{aligned}
\end{empheq}
Setting $\Tilde{F}=F-b(x,t)v_{x}+ \frac{b(x,t)a_x(x)}{a(x)}v$
    using the result of \cite[Lemma 2.4]{PAPER}, we obtain
    \begin{equation}
         \begin{aligned}
         & \int_{Q} \frac{1}{s\varphi}\left(a(x)\vert v_{xx}\vert^2+\frac{1}{ a(x)}\vert v_{t}\vert^2 \right)e^{2s\varphi}dx dt+\int_{Q} \lambda \vert v_{x}\vert^2 e^{2s\varphi} dx dt+\int_{Q} s\lambda^2 \varphi \frac{1}{a(x)}v^2 e^{2s\varphi} dxdt\\& \leqslant C \int_{Q}\frac{1}{a(x)}\vert \Tilde{F} \vert^2 e^{2s\varphi}dxdt +Cs\lambda \left(\varphi(T)\| v(\cdot,T)\|_{\frac{1}{a}}^2+\| v_{x}(\cdot,T)\|_{L^2(0,1)}^2\right)e^{2s\varphi(T)}\\& +C\left(s\lambda \| v(\cdot,0)\|_{\frac{1}{a}}^2 +\| v_{x}(\cdot,0)\|_{L^2(0,1)}^2\right)e^{2s}\\ & \leqslant C \int_{Q}\frac{1}{a(x)}\vert F-b(x,t)v_{x}+ \frac{b(x,t)a_x(x)}{a(x)}v \vert^2 e^{2s\varphi} dxdt +Cs\lambda \left(\varphi(T)\| v(\cdot,T)\|_{\frac{1}{a}}^2+\| v_{x}(\cdot,T)\|_{L^2(0,1)}^2\right)e^{2s\varphi(T)}\\& +C\left( s\lambda\| v(\cdot,0)\|_{\frac{1}{a}}^2 +\| v_{x}(\cdot,0)\|_{L^2(0,1)}^2\right)e^{2s}\\ & \leqslant C \int_{Q}\left\lbrace\frac{1}{a(x)}\vert F\vert^2 + \frac{b(x,t)^2}{a(x)}\vert v_{x}\vert^2+ \left(\frac{b(x,t)a_x(x)}{a(x)}\right)^2\frac{1}{a(x)}\vert  v\vert^2 \right\rbrace e ^{2s\varphi}dxdt \\&+Cs\lambda \left(\varphi(T)\| v(\cdot,T)\|_{\frac{1}{a}}^2+\| v_{x}(\cdot,T)\|_{L^2(0,1)}^2\right)e^{2s\varphi(T)} +C\left( s\lambda\| v(\cdot,0)\|_{\frac{1}{a}}^2 +\| v_{x}(\cdot,0)\|_{L^2(0,1)}^2\right)e^{2s}.
    \end{aligned}
    \end{equation}
    This achieves the proof.
\end{proof}
\subsection{Linearized MFG system}\label{Subsec2.3}

In this section, we combine the Carleman estimates of the above sections to obtain the following one for the linearized Mean-Field Game system
\begin{empheq}[left = \empheqlbrace]{alignat=2}
\begin{aligned}
& u_t(x,t) +a(x) u_{xx}(x,t)+d_{1}(x,t) u_{x}=   d_{2}(x,t) m + F, &&\quad \text { in } Q, \\
& m_{t}(x,t) -\left(a(x) m(x,t)\right)_{xx} +c_{1}(x,t)m_{x}=b(x,t) m+ c_{2}(x,t)u_{x} + \rho(x,t) u_{xx} + G, &&\quad \text { in } Q, \\
& u= am=0, &&\quad \text { on } \Sigma\\
& m(\cdot,0)=m_0,\quad u(\cdot,T)=h, &&\quad \text { in } \Omega. \label{eq1to155}
\end{aligned}
\end{empheq}

\begin{lemma}\label{33}
There exists a constant $\lambda_{0}>0$ such that for each $\lambda>\lambda_{0},$ we can find a constant $s_{0}(\lambda)>0$ so that there exists a constant $C>0$ such that
\begin{equation*}
    \begin{aligned}
        &\int_{Q}\left(\frac{1}{a(x)}\vert u_{t}\vert^2+ a(x)\vert u_{xx}\vert^2+ s\lambda \varphi \vert u_{x}\vert^2+ \frac{s^2\lambda^2\varphi^2}{a(x)}\vert u\vert^2\right)e^{2s\varphi}dxdt\\ &+\int_{Q}\frac{1}{s\varphi}\left(a(x)\vert (am)_{xx}\vert^2+a(x)\vert m_{t}\vert^2 \right)e^{2s\varphi}dx dt+\int_{Q} \lambda \vert (am)_{x}\vert^2 e^{2s\varphi} dx dt+\int_{Q} s\lambda^2 \varphi a(x)\vert m\vert^2 e^{2s\varphi} dxdt\\&\leqslant C  \int_{Q}s\varphi\frac{1}{a(x)}\vert Fe^{s\varphi}\vert^2 dxdt+C \int_{Q}a(x)\vert Ge^{s\varphi}\vert^2 dxdt + Cs\left( s\lambda\varphi(T)\| u(\cdot,T)\|_{\frac{1}{a}}^2+\| u(\cdot,T)\|_{1,\frac{1}{a}}^2 \right)e^{2s\varphi(T)} \\&+Cs\left( s\lambda\| u(\cdot,0)\|_{\frac{1}{a}}^2+\| u(\cdot,0)\|_{1,\frac{1}{a}}^2\right)e^{2s}  +C \left(s\lambda \varphi(T)\| m(\cdot,T)\|_{a}^2+\| m(\cdot,T)\|_{1,a}^2\right)e^{2s\varphi(T)}\\&+C\left( s\lambda \| m(\cdot,0)\|_{a}^2 +\| m(\cdot,0)\|_{1,a}^2\right)e^{2s}
    \end{aligned}
\end{equation*}
for all $s\geqslant s_{0}(\lambda)$ and $(u,m)\in H^2_{\frac{1}{a}}(0,1) \times H_{2,a}(0,1)$ satisfying  \eqref{eq1to155}.
\end{lemma}
\begin{proof}
    We set $\Tilde{F}=d_{2} m +F$ and $\Tilde{G}=b m+ c_{2}u_{x} + \rho u_{xx}+G$. We note that
    $$\vert \Tilde{F}(x,t)\vert \leqslant  \vert d_{2}(x,t)m(x,t)\vert +\vert F\vert,$$
    $$\vert \Tilde{G}(x,t)\vert \leqslant C \left( \vert m(x,t)\vert+\vert u_{x}\vert+ \vert u_{xx}\vert+\vert G \vert \right)$$
    for all $(x,t)\in Q$. Using $\vert d_{2}(x,t)\vert \leqslant C a(x)$ we obtain
   \begin{equation}\label{36}
        \begin{aligned}
           \int_{Q}s\varphi\frac{1}{a(x)}\vert \Tilde{F}e^{s\varphi}\vert^2 dxdt&\leqslant  C \left(\int_{Q}s\varphi\frac{d_{2}(x,t)^2}{a(x)}\vert m\vert^2e^{2s\varphi} dx dt
         +\int_{Q}s\varphi\frac{1}{a(x)}\vert F\vert^2 e ^{2s\varphi}dxdt\right)\\&\leqslant C\left(\int_{Q}s\varphi a(x)\vert m\vert^2e^{2s\varphi} dxdt+\int_{Q}s\varphi\frac{1}{a(x)}\vert F\vert^2 e ^{2s\varphi}dxdt\right).\\
    \end{aligned}
   \end{equation}
   Moreover, \eqref{Carlm HJB} implies that
  \begin{equation}\label{37}
     \begin{aligned}
          \int_{Q}a(x)\vert \Tilde{G}\vert^2e^{2s\varphi} dxdt &\leqslant C \int_{Q}a(x)\left(\vert m \vert^2 +\vert u_x\vert^2 +\vert u_{xx}\vert^2 \right)e^{2s\varphi}dxdt +C\int_{Q}a(x)\vert G\vert^2 e^{2s\varphi}dxdt \\&\hspace{-1.5cm}\leqslant C \int_{Q}\left(a(x)\vert m \vert^2 +C'\vert u_x\vert^2 \right)e^{2s\varphi}dxdt +C \int_{Q}a(x)\vert u_{xx}\vert^2e^{2s\varphi}dxdt+C\int_{Q}a(x)\vert G\vert^2e^{2s\varphi} dxdt\\&\hspace{-1.5cm} \leqslant C \int_{Q}\left(a(x)\vert m \vert^2 +C'\vert u_x\vert^2\right)e^{2s\varphi}dxdt+C  \int_{Q}s\varphi\frac{1}{a(x)}\vert \Tilde{F}\vert^2 e^{2s\varphi}dxdt\\&\hspace{-1.3cm} +Cs\left( s\lambda\varphi(T)\| u(\cdot,T)\|_{\frac{1}{a}}^2+\| u(\cdot,T)\|_{1,\frac{1}{a}}^2 \right)e^{2s\varphi(T)}\\&\hspace{-1.3cm} +Cs\left( s\lambda\| u(\cdot,0)\|_{\frac{1}{a}}^2+\| u(\cdot,0)\|_{1,\frac{1}{a}}^2\right)e^{2s} +C\int_{Q}a(x)\vert G\vert^2 e^{2s\varphi}dxdt.
     \end{aligned}
  \end{equation}
Adding \eqref{driftterm} and \eqref{Carlm FP}, and applying \eqref{36}-\eqref{37}, we obtain
  \begin{align*}
      &\int_{Q}\left\lbrace\left(\frac{1}{a(x)}\vert u_{t}\vert^2+ a(x)\vert u_{xx}\vert^2+ s\lambda \varphi \vert u_{x}\vert^2+ \frac{s^2\lambda^2\varphi^2}{a(x)}\vert u\vert^2\right)e^{2s\varphi}dxdt\right\rbrace\\ &+\int_{Q}\left\lbrace\frac{1}{s\varphi}\left(a(x)\vert (am)_{xx}\vert^2+a(x)\vert m_{t}\vert^2 \right)e^{2s\varphi}dx dt+\int_{Q} \lambda \vert (am)_{x}\vert^2 e^{2s\varphi} dx dt+\int_{Q} s\lambda^2 \varphi a(x)\vert m\vert^2 e^{2s\varphi} dxdt\right\rbrace\\& \leqslant C\int_{Q}s\varphi\frac{1}{a(x)}\vert \Tilde{F}\vert^2 e^{2s\varphi}dxdt+\int_{Q}a(x)\vert \Tilde{G}\vert^2e^{2s\varphi} dxdt+Cs\left( s\lambda\varphi(T)\| u(\cdot,T)\|_{\frac{1}{a}}^2+\| u(\cdot,T)\|_{1,\frac{1}{a}}^2 \right)e^{2s\varphi(T)}\\&\quad +Cs\left( s\lambda\| u(\cdot,0)\|_{\frac{1}{a}}^2+\| u(\cdot,0)\|_{1,\frac{1}{a}}^2\right)e^{2s}+C\left(s\lambda \varphi(T)\| m(\cdot,T)\|_{a}^2+\| (am)_{x}(\cdot,T)\|_{2}^2\right)e^{2s\varphi(T)}\\& \quad +C\left( s\lambda \| m(\cdot,0)\|_{a}^2 +\| (am)_{x}(\cdot,0)\|_{2}^2\right)e^{2s} \\&\leqslant C\int_{Q}s\varphi a(x)\vert m\vert^2e^{2s\varphi} dxdt + C\int_{Q}s\varphi\frac{1}{a(x)}\vert F\vert^2 e^{2s\varphi}dxdt+ C \int_{Q}\left(a(x)\vert m \vert^2 +C'\vert u_x\vert^2  \right)e^{2s\varphi}dxdt\\
      &\quad +C\int_{Q}a(x)\vert G\vert^2 e ^{2s\varphi} dxdt+Cs\left( s\lambda\varphi(T)\| u(\cdot,T)\|_{\frac{1}{a}}^2+\| u(\cdot,T)\|_{1,\frac{1}{a}}^2 \right)e^{2s\varphi(T)}\\
      &\quad +Cs\left( s\lambda\| u(\cdot,0)\|_{\frac{1}{a}}^2+\| u(\cdot,0)\|_{1,\frac{1}{a}}^2\right)e^{2s} +C \left(s\lambda \varphi(T)\| m(\cdot,T)\|_{a}^2+\| (am)_{x}(\cdot,T)\|_{2}^2\right)e^{2s\varphi(T)}\\
      &\quad +C\left( s\lambda \| m(\cdot,0)\|_{a}^2 +\| (am)_{x}(\cdot,0)\|_{2}^2\right)e^{2s}.
  \end{align*}
This yields the desired estimate.
\end{proof}
\section{Stability of the backward problem for linearized MFG system} \label{sec3}In this section, we prove H\"older and logarithmic stability estimates for the backward problem of linearized MFG systems. We start with the case $0<t_{0}<T$.
\begin{theorem}\label{38}
Let $(u,m)\in H^2_{\frac{1}{a}}(0,1)\times H_{2,a}(0,1)$ satisfy \eqref{eq1to155} with 
$$\| u(\cdot,0)\|_{1,\frac{1}{a}}\leqslant M \;\;\; \text{ and }\;\;\;\| m(\cdot,0)\|_{1,a}\leqslant M.$$
%\textcolor{blue}{We treat two cases:\\
%{\bf Case $0<t_0 < T$:}}\\
There exist constants $C>0$ and $\theta \in (0,1)$ depending on $t_{0}\in (0,T)$ and $M$ such that 
\begin{align*}
    \| u(\cdot,t_{0})\|_{\frac{1}{a}}+\| m(\cdot,t_{0})\|_{a}\leqslant C\left( D_{0}^{\theta}+D_{0}\right),
\end{align*}
where 
$$ D_{0}=\| u(\cdot,T)\|_{1,\frac{1}{a}}+\| m(\cdot,T)\|_{1,a}.
$$
\end{theorem}
\begin{proof}
Using our Carleman estimate and the fact that $\varphi(t_{0})\leqslant \varphi(t)$ for $t_{0}\leqslant t\leqslant T,$ yields
    \begin{align*}
        e^{2s\varphi(t_{0})}\int_{(0,1)\times (t_{0},T)} \left(\frac{1}{a(x)}\vert u_{t}\vert ^2 + \frac{s^2\lambda^2 \varphi^2}{a(x)}u^2\right)dx dt \leqslant &Cs^2\lambda \varphi(T)\| u(\cdot,T)\|_{1,\frac{1}{a}}^2 e^{2s\varphi(T)}+Cs^2\lambda M^2 e^{2s}\\ &+Cs\lambda \varphi(T)\| m(\cdot,T)\|_{1,a}^2e^{2s\varphi(T)}+Cs\lambda M^2e^{2s},
    \end{align*}
    and 
    \begin{align*}
       e^{2s\varphi(t_{0})}  \int_{(0,1)\times (t_{0},T)}\left(\frac{a(x)}{s\varphi}\vert m_{t}\vert^2 +s\lambda^2 \varphi a(x)\vert m\vert^2  \right)dxdt \leqslant &Cs^2\lambda \varphi(T)\| u(\cdot,T)\|_{1,\frac{1}{a}}^2 e^{2s\varphi(T)}+Cs^2\lambda M^2 e^{2s}\\ &+Cs\lambda \varphi(T)\| m(\cdot,T)\|_{1,a}^2e^{2s\varphi(T)}+Cs\lambda M^2e^{2s}.
    \end{align*}
    For $\lambda >0$ sufficiently large, we have

      \begin{align*}
       \int_{(0,1)\times (t_{0},T)} \left(\frac{1}{a(x)}\vert u_{t}\vert ^2 +\frac{s^2\lambda^2 \varphi^2 }{a(x)}u^2\right)&dx dt \leqslant Cs^2\| u(\cdot,T)\|_{1,\frac{1}{a}}^2 e^{2s(\varphi(T)-\varphi(t_{0}))}+Cs^2 M^2  e^{-2s\alpha(t_{0})}\\ &+Cs\| m(\cdot,T)\|_{1,a}^2e^{2s(\varphi(T)-\varphi(t_{0}))}+Cs M^2 e^{-2s\alpha(t_{0})},
    \end{align*}
    and
       \begin{align*}
     \int_{(0,1)\times (t_{0},T)}\left(\frac{a(x)}{s}\vert m_{t}\vert^2 +s\lambda^2 \varphi a(x)\vert m\vert^2  \right)&dxdt \leqslant Cs^2\| u(\cdot,T)\|_{1,\frac{1}{a}}^2 e^{2s(\varphi(T)-\varphi(t_{0}))}+Cs^2 M^2 e^{-2s\alpha(t_{0})}\\ &+Cs\| m(\cdot,T)\|_{1,a}^2e^{2s(\varphi(T)-\varphi(t_{0}))}+CsM^2e^{-2s\alpha(t_{0})},
    \end{align*}
    with $\alpha(t_{0}):=\varphi(t_{0})-1$ and the constant depends on $T$ and $\lambda$. Next,
    \begin{equation}\label{39}
        \begin{aligned}
            \| u_{t}\|_{L^{2}(t_{0},T, L^2_{\frac{1}{a}}(0,1))}^2 \leqslant Cs^2D_{0}^2e^{2s(\varphi(T)-\varphi(t_{0}))}+Cs^2 M^2 e^{-2s\alpha(t_{0})},
        \end{aligned}
    \end{equation}
    and 
       \begin{equation}\label{40}
        \begin{aligned}
            \| m_{t}\|_{L^{2}(t_{0},T,L^2_{a}(0,1))}^2 \leqslant Cs^3D_{0}^2e^{2s(\varphi(T)-\varphi(t_{0}))}+Cs^3 M^2 e^{-2s\alpha(t_{0})}.
        \end{aligned}
    \end{equation}
    Seeing that for $\psi=u,m$, 
    $$ \psi(x,t_{0})=\int_{T}^{t_{0}}\psi_{t}(x,t)dt +\psi(x,T)\quad  \forall x \in (0,1),$$
    there is a constant $C_{6}>0$ such that for all $t_{0}\in (0,T)$ we have
    \begin{align}\label{41}
    \| u(\cdot,t_{0})\|_{\frac{1}{a}}^2\leqslant C_{6} \| u_{t}\|_{L^{2}(t_{0},T, L^2_{\frac{1}{a}}(0,1))}^2+ C_{6} \| u(\cdot,T)\|^2_{\frac{1}{a}},
\end{align}
and there exists a constant $C'_{6}>0$ such that for all $t_{0}\in (0,T)$ we have
  \begin{align}\label{42}
    \| m(\cdot,t_{0})\|_{a}^2\leqslant C'_{6} \| m_{t}\|_{L^{2}(t_{0},T, L^2_{a}(0,1))}^2+ C'_{6} \| m(\cdot,T)\|^2_{a}.
\end{align}
For the inequality \eqref{41}, we estimate the first term on the right-hand side by \eqref{39} and using $\varphi(T)>1$. We obtain
\begin{equation}\label{43}
    \begin{aligned}
        \| u(\cdot,t_{0})\|_{\frac{1}{a}}^2&\leqslant C_{6}\left(Cs^2D_{0}^2e^{2s(\varphi(T)-\varphi(t_{0}))}+Cs^2 M^2 e^{-2s\alpha(t_{0})}\right)+C_{6} \| u(\cdot,T)\|^2_{\frac{1}{a}}
        \\&\leqslant 
C_{7}s^2D_{0}^2e^{2s\varphi(T)}+C_{7}s^2M^2e^{-2s\alpha(t_{0})} \\ & \leqslant C_{7}D_{0}^2e^{3s\varphi(T)} +C_{7}M^2e^{-s\alpha(t_{0})}
    \end{aligned}
\end{equation}
for all $s>s_{0}$ and all $t_{0}\in (0,T).$ For the inequality \eqref{42}, we estimate the first term on the right-hand side by \eqref{40}; we obtain
\begin{equation}\label{44}
    \begin{aligned}
        \| m(\cdot,t_{0})\|_{a}^2 &\leqslant C'_{6}\left(Cs^3D_{0}^2e^{2s(\varphi(T)-\varphi(t_{0}))}+Cs^3 M^2 e^{-2s\alpha(t_{0})}\right)+C'_{6} \| m(\cdot,T)\|^2_{a}\\& \leqslant 
C'_{7}s^3D_{0}^2e^{2s\varphi(T)}+C'_{7}s^3M^2e^{-2s\alpha(t_{0})} \\ & \leqslant C'_{7}D_{0}^2e^{3s\varphi(T)} +C'_{7}M^2e^{-s\alpha(t_{0})}
    \end{aligned}
\end{equation}
for all $s>s_{0}$ and all $t_{0}\in (0,T).$ Therefore, 
\begin{align}\label{45}
   \| u(\cdot,t_{0})\|_{\frac{1}{a}}^2 \leqslant C_{8}D_{0}^2e^{3s\varphi(T)} +C_{8}M^2e^{-s\alpha(t_{0})} \;\;\; \text{for all } \;\; s>0,
\end{align}
and 
\begin{align}\label{46}
    \| m(\cdot,t_{0})\|_{a}^2 \leqslant C'_{8}D_{0}^2e^{3s\varphi(T)} +C'_{8}M^2e^{-s\alpha(t_{0})} \;\;\; \text{for all } \;\; s>0,
\end{align}
with $C_{8}=C_{7}e^{6s_{0}\varphi(T)}$ and $C'_{8}=C'_{7}e^{6s_{0}\varphi(T)}$. Indeed, For $s\geqslant s_{0}$,   the inequalities are trivially verified.
For $s<s_{0}$,   we have $2s_{0}-s>s_{0}$ and applying \eqref{43} and \eqref{44} for $2s_{0}-s$,  we obtain 
\begin{align*}
   \| u(\cdot,t_{0})\|_{\frac{1}{a}}^2 &\leqslant C_{7}D_{0}^2e^{3(2s_{0}-s)\varphi(T)} +C_{7}M^2e^{-(2s_{0}-s)\alpha(t_{0})} ,
\end{align*}
and 
\begin{align*}
   \| m(\cdot,t_{0})\|_{a}^2 &\leqslant C_{7}D_{0}^2e^{3(2s_{0}-s)\varphi(T)} +C_{7}M^2e^{-(2s_{0}-s)\alpha(t_{0})} .
\end{align*}
Hence,  we obtain 
\begin{align*}
   \| u(\cdot,t_{0})\|_{\frac{1}{a}}^2 \leqslant C_{8}D_{0}^2e^{3s\varphi(T)} +C_{8}M^2e^{-s\alpha(t_{0})} \;\;\; \text{for all } \;\; s<s_0,
\end{align*}
and 
\begin{align*}
    \| m(\cdot,t_{0})\|_{a}^2 \leqslant C'_{8}D_{0}^2e^{3s\varphi(T)} +C'_{8}M^2e^{-s\alpha(t_{0})} \;\;\; \text{for all } \;\;  s<s_0, 
\end{align*}
with $C_{8}=C_{7}e^{6s_{0}\varphi(T)}$ and $C'_{8}=C'_{7}e^{6s_{0}\varphi(T)}$.
Next,
\begin{align*}
    \| u(\cdot,t_{0})\|_{\frac{1}{a}}^2 +\| m(\cdot,t_{0})\|_{a}^2 \leqslant C_{9}D_{0}^2e^{3s\varphi(T)} +C_{9}M^2e^{-s\alpha(t_{0})} \;\;\; \text{for all } \;\; s>0.
\end{align*}
To obtain the desired result, we study two possible cases
\begin{itemize}
    \item If $M\leqslant D_{0}$, 
    we take $s=0$ so that the inequalities \eqref{45} and \eqref{46} yield 
    $$ \| u(\cdot,t_{0})\|_{\frac{1}{a}}^2 +\| m(\cdot,t_{0})\|_{a}^2 \leqslant 2 C_{9}D_{0}^2.$$
    \item  If $M>D_{0}$, 
    we choose $s>0$ such that 
    $$ D_{0}^2e^{3s\varphi(T)}=M^2e^{-s\alpha(t_{0})}.$$
Take
$$ s= \frac{2}{3\varphi(T)+\alpha(t_{0})}\log\frac{M}{D_{0}}>0,$$
so that 
$$ \| u(\cdot,t_{0})\|_{\frac{1}{a}}^2 +\| m(\cdot,t_{0})\|_{a}^2 \leqslant 2 C_{9}M^{\frac{6\varphi(T)}{3\varphi(T)+\alpha(t_{0})}}D_{0}^{\frac{2\alpha(t_{0})}{3\varphi(T)+\alpha(t_{0})}}.
$$
Taking $\theta:=\frac{\alpha(t_{0})}{3\varphi(T)+\alpha(t_{0})}\in (0,1),$ we obtain the result.
\end{itemize}
This completes the proof.
\end{proof}
We end this section with a logarithmic stability estimate in the case $t_{0}=0$. 

\begin{theorem}\label{51}
Let $(u,m)\in H^2_{\frac{1}{a}}(0,1)\times H_{2,a}(0,1)$ satisfy \eqref{eq1to155} with 
\begin{align*}
    &\sum_{k=0}^2\| \partial_{t}^k u(\cdot,0) \|_{1,\frac{1}{a}}\leqslant M,\quad 
    \sum_{k=0}^2\| \partial_{t}^k m(\cdot,0) \|_{1,a}\leqslant M,
\end{align*}
with an arbitrarily chosen constant $M>0.$
%\textcolor{blue}{We treat two cases:\\
%{\bf Case $0<t_0 < T$:}}\\
 Then for any $\alpha\in (0,1)$, there exists a constant $C>0$ such that 
\begin{align*}
    \| u(\cdot,0)\|_{\frac{1}{a}}+  \| m(\cdot,0)\|_{a}\leqslant C\left(\log\frac{1}{D}\right)^{-\alpha},
\end{align*}
where 
$$ D=\sum_{k=0}^2\| \partial_{t}^k u(\cdot,T) \|_{1,\frac{1}{a}} +\sum_{k=0}^2\| \partial_{t}^k m(\cdot,T) \|_{1,a}.$$

\end{theorem}
\begin{proof}
To simplify the proof, we assume that the coefficients $d_1, \, d_2,\, c_1, \, c_2, \, b$ are independent of $t$. Otherwise, we can impose additional regularity on these coefficients.
    The proof is based on the stability for $t_{0}>0$ and for $\psi=u, m$
    $$ \psi(x,0)=\int_{T}^0 \partial_{t}\psi(x,t)dt + \psi(x,T), \;\; x \in (0,1).$$
    Then arguing similarly to the proof of Theorem \ref{38} to $\partial_{t}^2u$, we obtain
    \begin{align*}
        \int_{(0,1)\times (t_{0},T)} \left(\frac{1}{a(x)}\vert\partial_{t}^3 u\vert ^2 + \frac{s^2\lambda^2 \varphi^2}{a(x)}\vert\partial_{t}^2u\vert^2\right)dx dt \leqslant Cs^2D^2 e^{2s\varphi(T)}+Cs^2M^2 e^{-2s\alpha(t_{0})},
    \end{align*}
    and
    \begin{align*}
         \int_{(0,1)\times (t_{0},T)}\left(\frac{a(x)}{s}\vert \partial_{t}^3m\vert^2 + s\lambda^2 \varphi a(x)\vert\partial_{t}^2m\vert^2 \right)dx dt \leqslant C s^2 D^2 e^{2s\varphi(T)}+Cs^2M^2 e^{-2s\alpha(t_{0})}.
    \end{align*}
    Next,
    
    \begin{equation*}\label{52}
        \begin{aligned}
              \|\partial_{t}^2 u\|^2_{L^2(t_{0},T,L^2_{\frac{1}{a}}(0,1))}&\leqslant C D^2e^{2s\varphi(T)}+CM^2e^{-2s\alpha(t_{0})}\\& \leqslant C D^2 e^{3s\varphi(T)}+CM^2e^{-s\alpha(t_{0})}
        \end{aligned}
    \end{equation*}
   for all $s\geqslant s_0$, and 
    
     \begin{equation*}\label{53}
        \begin{aligned}
            \|\partial_{t}^2 m\|^2_{L^2(t_{0},T,L^2_{a}(0,1))}&\leqslant C sD^2e^{2s\varphi(T)}+CsM^2e^{-2s\alpha(t_{0})}\\ & \leqslant C D^2 e^{3s\varphi(T)}+CM^2e^{-s\alpha(t_{0})}
        \end{aligned}
    \end{equation*}
    for all $s\geqslant s_0$. Since for $\psi=u,m$,
    $$ \partial_t \psi(x,t_{0})=\int_{T}^{t_{0}}\partial_{t}^2\psi(x,t)dt+\partial_{t}\psi(x,T), \quad x\in (0,1),$$
    we have
    \begin{align*}
        \| \partial_t u(\cdot,t_{0})\|_{\frac{1}{a}}^2\leqslant C\| \partial_t^2u\|^2_{L^2(t_{0},T,L^2_{\frac{1}{a}}(0,1))}+C  \| \partial_t u(\cdot,T)\|_{\frac{1}{a}}^2,
    \end{align*}
    and 
    \begin{align*}
        \| \partial_t m(\cdot,t_{0})\|_{a}^2\leqslant C\| \partial_t^2m\|^2_{L^2(t_{0},T,L^2_{a}(0,1))}+C  \| \partial_t m(\cdot,T)\|_{a}^2.
    \end{align*}
    Then,
    \begin{equation}\label{54}
        \begin{aligned}
             \|  \partial_{t}u(\cdot,t_{0})\|_{\frac{1}{a}}^2&\leqslant  C D^2 e^{3s\varphi(T)}+CM^2e^{-s\alpha(t_{0})}+ C  \| \partial_t u(\cdot,T)\|_{\frac{1}{a}}^2\\& \leqslant C D^2 e^{3s\varphi(T)}+CM^2e^{-s\alpha(t_{0})}
        \end{aligned}
    \end{equation}
    for all $s\geqslant s_0$, and 
        \begin{equation}\label{55}
        \begin{aligned}
             \| \partial_{t}m(\cdot,t_{0})\|_{a}^2&\leqslant  C D^2 e^{3s\varphi(T)}+CM^2e^{-s\alpha(t_{0})}+ C  \| \partial_t m(\cdot,T)\|_{a}^2\\& \leqslant C D^2 e^{3s\varphi(T)}+CM^2e^{-s\alpha(t_{0})}
        \end{aligned}
    \end{equation}
     for all $s\geqslant s_0$.  In the previous two inequalities, we used  
     $$\| \partial_t m(\cdot,T)\|_{a}^2 + \| \partial_t u(\cdot,T)\|_{\frac{1}{a}}^2 \leqslant D^2 \leqslant C D^2 e^{3s\varphi(T)}.
     $$
     Using, for $\psi=u,m$,
     $$ \psi(x,0)=\int_{T}^0\partial_{t}\psi(x,\nu)d\nu+\psi(x,T), \quad x\in (0,1),$$
     we obtain
     \begin{align*}
         \int_{0}^1 \frac{1}{a(x)}\vert u(x,0)\vert^2dx\leqslant &  2 \int_{0}^1 \frac{1}{a(x)}\left\vert \int_{T}^0u_{t}(x,\nu)d\nu\right\vert^2 dx +2 \int_{0}^1 \frac{1}{a(x)}\vert u(x,T)\vert^2dx \\ &\leqslant C \int_{0}^{T} \| u_{t}(\cdot,\nu)\|_{\frac{1}{a}}^2d\nu+ C \| u(\cdot,T)\|_{\frac{1}{a}}^2,
     \end{align*}
     and 
     \begin{align*}
         \int_{0}^1 a(x)\vert m(x,0)\vert^2dx\leqslant &  2 \int_{0}^1 a(x)\left\vert \int_{T}^0m_{t}(x,\nu)d\nu\right\vert^2 dx +2 \int_{0}^1 a(x)\vert m(x,T)\vert^2dx \\ &\leqslant C \int_{0}^{T} \| m_{t}(\cdot,\nu)\|_{a}^2d\nu+ C \| m(\cdot,T)\|_{a}^2.
     \end{align*}
    According to \eqref{54}, we obtain
    \begin{equation*}
        \begin{aligned}
            \| u(\cdot,0)\|_{\frac{1}{a}}^2&\leqslant C_{1} D^2\int_{0}^Te^{3s\varphi(T)}d\nu+C_{1}M^2\int_0^Te^{-s\alpha(\nu)}d\nu + C \| u(\cdot,T)\|_{\frac{1}{a}}^2\\& \leqslant C_{1}D^2e^{C_2s}+C_{1}M^2\int_0^Te^{-s\alpha(\nu)}d\nu,
        \end{aligned}
    \end{equation*}
and using \eqref{55}, we deduce
    \begin{equation*}
        \begin{aligned}
            \| m(\cdot,0)\|_{a}^2&\leqslant C_{1} D^2\int_{0}^Te^{3s\varphi(T)}d\nu+C_{1}M^2\int_0^Te^{-s\alpha(\nu)}d\nu + C \| m(\cdot,T)\|_{a}^2\\& \leqslant C_{1}D^2e^{C_2s}+C_{1}M^2\int_0^Te^{-s\alpha(\nu)}d\nu.
        \end{aligned}
    \end{equation*}
Then,
    \begin{equation}\label{56}
        \begin{aligned}
            \| u(\cdot,0)\|_{\frac{1}{a}}^2+  \| m(\cdot,0)\|_{a}^2\leqslant C_{1}D^2e^{C_2s}+C_{1}M^2\int_0^Te^{-s\alpha(\nu)}d\nu.
        \end{aligned}
    \end{equation}
We calculate  this integral $\int_{0}^T e^{-s\alpha(\nu)}d\nu$, using the following change of variable $\tau = \alpha(\nu)= e^{\lambda \nu}-1$
$$ \int_{0}^T e^{-s\alpha(\nu)}d\nu=\frac{1}{\lambda} \int_{0}^Te^{-s\tau}\frac{1}{1+\tau}d\tau \leqslant \frac{1}{\lambda}\left[\frac{e^{-s\tau}}{s}\right]_{\tau =e^{\lambda T}-1}^{\tau=0}\leqslant \frac{1}{s\lambda}.$$
Hence, \eqref{56} yields
\begin{align}\label{57}
        \| u(\cdot,0)\|_{\frac{1}{a}}^2+  \| m(\cdot,0)\|_{a}^2\leqslant C_{3}D^2e^{C_2s}+\frac{C_3}{s}M^2
\end{align}
for all $s\geqslant s_0.$ Then we obtain \eqref{57} for all $s>0$ by replacing $s:=s+s_{0}.$
We suppose that $D<1$. Setting 
$$ s=\left(\log\frac{1}{D}\right)^{\alpha}>0,$$
with $0<\alpha<1$, we have
\begin{align*}
    e^{C_{2}s}D^2=&\exp\left(C_{2}\left(\log\frac{1}{D}\right)^{\alpha}\right) D^2\\ &
    =\exp\left(-2\left(\log\frac{1}{D}\right)+ C_{2}\left(\log\frac{1}{D}\right)^{\alpha}\right).
\end{align*}
Moreover, there exists a constant $C_{4}>0$ such that $$ e^{-2\xi+C_{2}\xi^{\alpha}}\leqslant \frac{C_{4}}{\xi^{\alpha}}\;\;\;\; \text{for all}\;\; \xi >0.$$
Hence, 
$$e^{C_{2}s}D^2 \leqslant C_{4}\left(\log\frac{1}{D}\right)^{-\alpha} ,$$
 and 
$$C_{3}D^2e^{C_{2}s}+ \frac{C_{3}}{s}M_{1}^2 \leqslant C'C_{4}\left(\log\frac{1}{D}\right)^{-\alpha}+C_{3}M_{1}^2\left(\log\frac{1}{D}\right)^{-\alpha}.$$
Thus,  we have the result.
\end{proof}

\section{Stability of the backward problem for nonlinear MFG system}\label{sec4}
Now we prove the main stability result for the backward problem of degenerate nonlinear MFG systems by combining the Carleman estimates established in section \ref{Subsec2.3} for the linearized system.
We recall the degenerate MFG system
\begin{empheq}[left = \empheqlbrace]{alignat=2}
\begin{aligned}\label{eq1to10}
& u_t(x,t) +a(x) u_{xx}(x,t)-\frac{p(x,t) }{2}\vert u_{x}\vert^2+   d(x,t) m = F(x,t), &&\quad \text { in } Q,\\
& m_{t}(x,t) -\left(a(x) m(x,t)\right)_{xx} -\left( p(x,t)mu_{x}\right)_{x}= G(x,t), &&\quad \text { in } Q,\\
&  u= am=0, &&\quad \text { on } \Sigma,
\end{aligned}
\end{empheq}
where we assume that $p,p_{x},d , c\in L^{\infty}(Q)$. We further assume the following assumptions
$$\vert p(x,t)\vert \leqslant C \sqrt{a(x)} ,\;\;\;\; \vert d(x,t)\vert \leqslant Ca(x) ,\;\;\;\;\vert a_x(x)\vert \leqslant C \sqrt{a(x)}. 
$$
Now, we state the main stability result.
\begin{theorem} Let $t_{0}\in (0,T)$ and $(u_{k},m_{k})\in H^2_{\frac{1}{a}}(0,1)\times H_{2,a}(0,1)$, $k=1,2,$ satisfy \eqref{eq1to10} such that  
$$
\| u_{k}(\cdot,0)\|_{1,\frac{1}{a}}\leqslant M, \qquad \| m_{k}(\cdot,0)\|_{1,a}\leqslant M,
$$
and  
$$ \partial_{x}^l u_{k} \in L^{\infty}(Q),  \;\;\partial_{x}^n m_{k} \in L^{\infty}(Q),\;\; l=0,1,2,\;\; n=0,1.
$$
 Then there exist  constants $C>0$ and $\theta \in (0,1)$ depending on $t_{0}$ and $M$ such that 
\begin{align*}
   \| u_{2}(\cdot,t_0)-u_{1}(\cdot,t_0)\|_{\frac{1}{a}}+ \| m_{2}(\cdot,t_{0})-m_{1}(\cdot,t_{0})\|_{a}\leqslant C\left( D_{0}^{\theta}+D_{0}\right),
\end{align*}
where 
$$ D_{0}=\| u_{2}(\cdot,T)-u_{1}(\cdot,T)\|_{1,\frac{1}{a}}+\| m_{2}(\cdot,T)-m_{2}(\cdot,T)\|_{1,a}.
$$
\end{theorem}
\begin{proof}
Setting $u=u_{2}-u_{1}$ and $m=m_{2}-m_{1},$ we have
\begin{empheq}[left = \empheqlbrace]{alignat=2}
\begin{aligned}\label{eq1to111}
&  u_t +a(x) u_{xx}- \frac{p(x,t)}{2}\left(u_{1x}+u_{2x}\right)u_{x}+  d(x,t) m = 0, \qquad \text{ in } Q,\\
&  m_{t} -\left(a(x) m\right)_{xx} -\left(pm_{2} u_{xx}+(p_{x} m_{2}+pm_{2x})u_{x}+(pu_{1xx}+p_xu_{1x})m+pm_xu_{1x}\right)= 0, \;\text { in } Q ,\\
&  u= am=0, \quad \text { on } \Sigma .  
\end{aligned}
\end{empheq}
Setting $$d_{1}=\frac{p}{2}\left(u_{1x}+u_{2x}\right), \qquad c_1=pu_{1x},$$
and 
$$b=pu_{1xx}+p_xu_{1x}, \qquad c_{2}=p_{x} m_{2}+pm_{2x},  \qquad \rho=pm_{2},$$
we have $ b,c_2 \in L^{\infty}(Q)$ and 
$$\vert d_{1}(x,t) \vert\leqslant \left\vert  \frac{p(x,t)}{2}\left(u_{1x}+u_{2x}\right)\right\vert\leqslant C \sqrt{a(x)},$$
$$ \vert c_1(x,t) \vert= \vert pu_{1x}\vert \leqslant C \sqrt{a(x)}, \qquad  \vert d(x,t)\vert \leqslant Ca(x), \qquad  \vert \rho(x,t)\vert \leqslant C.$$
Hence, applying Lemma \ref{33}, we have
\begin{equation*}
    \begin{aligned}
        &\int_{Q}\left\lbrace\left(\frac{1}{a(x)}\vert u_{t}\vert^2+ a(x)\vert u_{xx}\vert^2+ s\lambda \varphi \vert u_{x}\vert^2+ \frac{s^2\lambda^2\varphi^2}{a(x)}\vert u\vert^2\right)e^{2s\varphi}dxdt\right\rbrace\\ &+\int_{Q}\left\lbrace\frac{1}{s\varphi}\left(a(x)\vert (am)_{xx}\vert^2+a(x)\vert m_{t}\vert^2 \right)e^{2s\varphi}dx dt+\int_{Q} \lambda \vert (am)_{x}\vert^2 e^{2s\varphi} dx dt+\int_{Q} s\lambda^2 \varphi a(x)\vert m\vert^2 e^{2s\varphi} dxdt\right\rbrace\\&\leqslant Cs\left( s\lambda\varphi(T)\| u(\cdot,T)\|_{\frac{1}{a}}^2+\| u(\cdot,T)\|_{1,\frac{1}{a}}^2 \right)e^{2s\varphi(T)} \\&+Cs\left( s\lambda\| u(\cdot,0)\|_{\frac{1}{a}}^2+\| u(\cdot,0)\|_{1,\frac{1}{a}}^2\right)e^{2s}  +C \left(s\lambda \varphi(T)\| m(\cdot,T)\|_{a}^2+\| m(\cdot,T)\|_{1,a}^2\right)e^{2s\varphi(T)}\\&+C\left( s\lambda \| m(\cdot,0)\|_{a}^2 +\| m(\cdot,0)\|_{1,a}^2\right)e^{2s}.
    \end{aligned}
\end{equation*}
Now similar arguments as in the proof of Theorem \ref{38} complete the proof. 
\end{proof}

\section{Conclusion and perspectives}\label{sec5}
The problem discussed in this paper is an inverse backward problem for second-order degenerate Mean-Field Game systems in one dimension. That is the determination of intermediate states from the final output data. We established conditional stability of Hölder and logarithmic rates by the Carleman estimates approach on suitable $L^2$-weighted spaces. This has been achieved using suitable assumptions on the coefficients and employing the weight function $e^{\lambda t}$ with the second large parameter $\lambda>0$ which is crucial in absorbing bad terms. It is worth mentioning that inverse problems for Mean-Field Game systems with degenerate diffusion are not well studied along the lines of our paper. Future works will investigate some numerical aspects of such problems. Although we have focused on Dirichlet boundary conditions, our results extend to Neumann boundary conditions by similar techniques. Finally, possible extensions to backward problems of the 2D MFG systems with degenerate diffusion in a rectangular domain deserve further investigation.

\end{document}